\documentclass[12pt]{amsart}
\usepackage{amscd}
\usepackage{mathrsfs}
\usepackage{a4wide}
\usepackage{color}
\usepackage{amssymb,amsmath,amsthm,amsfonts,latexsym}
\usepackage[utf8x]{inputenc}
\usepackage{bbm} %this is so we can use mathbb with numbers, but using the command bbm

\newtheorem{theorem}{Theorem}

\newtheorem{claim}[theorem]{Claim}

\newtheorem{corollary}[theorem]{Corollary}

\newtheorem{definition}[theorem]{Definition}

\newtheorem{lemma}[theorem]{Lemma}

\newtheorem{observation}[theorem]{Observation}

\newtheorem{proposition}[theorem]{Proposition}
\newtheorem{remark}[theorem]{Remark}

\newtheorem{question}[theorem]{Question}
\newtheorem*{unnumberedquestion}{Question}
\date{}
\begin{document}

\title{Ramsey theory for highly connected monochromatic subgraphs \\ F1814}
\author{Jeffrey Bergfalk,
Michael Hru\v{s}\'{a}k,
\and Saharon Shelah}

\address{Centro de Ciencas Matem\'aticas\\
UNAM\\
A.P. 61-3, Xangari, Morelia, Michoac\'an\\
58089, M\'exico}

\email{jeffrey@matmor.unam.mx}

\email{michael@matmor.unam.mx}

\address{Department of Mathematics, Rutgers University, Hill Center, Piscataway, New Jersey, U.S.A. 08854-8019}

\address{Institute of Mathematics, Hebrew University, Givat Ram, Jerusalem 91904, Israel}

\email{shelah@math.rutgers.edu}

\thanks{{\it Date:} 15 December, 2018.\newline
{\it 2010 MSC.} 03E02,03E10\newline
{\it Key words and phrases.} Ramsey Theory, $k$-connected graph, highly connected graph, Mahlo cardinal, weakly compact cardinal.\newline
{The research of the second author was supported  by a PAPIIT grant IN100317. The research of the third author was supported by the United States-Israel Binational Science Foundation (BSF grant no. 2010405), by the NSF grant no. NSF-DMS 1101597, and by the European Research Council Grant Number 338821.}}

\begin{abstract} An infinite graph is highly connected if the complement of any subgraph of smaller size is connected. We consider weaker versions of Ramsey's Theorem asserting that in any coloring of the edges of a complete graph there exist large highly connected subgraphs all of whose edges are colored by the same color.
\end{abstract}
\maketitle
Ramsey's celebrated theorem in its most basic infinite form is the following: for any partition of the collection of pairs of natural numbers into finitely many sets, there exists some infinite $X\subseteq\mathbb{N}$ whose pairs all fall in one of those sets. A pithier rendering is by way of Erd\H{o}s and Rado's arrow notation:
\[\tag{$*$}\aleph_0\rightarrow(\aleph_0)_k^2\;\textnormal{ for any finite $k$}.\]
Here the outer cardinals $\aleph_0$, $2$, and $k$ parametrize the sorts of partitions under consideration: letting $[\mu]^\lambda$ denote the size-$\lambda$ subsets of $\mu$, the partitions in question in the above relation are all of the form $c:[\aleph_0]^2\to k$. The cardinal inside the parentheses records how large a homogeneous set we seek with respect to any such partition, and the arrow tells us we can always find one.

The following points are basic to the theory:
\begin{enumerate}
\item The above relation ``descends'' to finite contexts (see \cite{graham} \textsection 1.5 for a direct deduction). More precisely, for any finite $k$ and $m$ there is some $n$ such that
$$n\to(m)_k^2$$
This least such $n$ is often denoted $R(m;k)$.
\item Extending the relation ($*$) to higher cardinals is less straightforward. More precisely:
\begin{enumerate}
\item The relation $\aleph_1\rightarrow(\aleph_1)_k^2$ fails in a very strong sense, for any $k\leq\aleph_0$ (see \cite{todorcevic-pairs}).
\item More generally, for $\kappa<\mu$, the relation $\mu\rightarrow(\mu)_\kappa^2$ characterizes any uncountable cardinal $\mu$ as \textit{weakly compact}, that is, as a cardinal whose existence is a strictly stronger assumption than the ZFC axioms (see \cite{kanamori}).
\item Item (1), on the other hand, does fully generalize: for any $\kappa$ and $\mu$ there's some least $\nu$ such that 
$$\nu\to(\mu)_\kappa^2$$
This follows from Erd\H{o}s and Rado's theorem that $(2^\lambda)^{+}\rightarrow(\lambda^{+})_\lambda^2$ for any infinite $\lambda$ (see \cite{erdos-rado}).
\end{enumerate}
\end{enumerate}
Each of these facts will figure in the following. Recall lastly the more pictorial framing of Ramsey's relation in terms of edge-colorings of graphs:
$$\nu\to(\mu)_\lambda^2$$
if and only if every coloring of the edges of the complete graph on $\nu$ by $\lambda$ many colors contains some size-$\mu$ monochromatic subgraph which is complete.
It is this framing we will have generally in mind --- only our interest will be in subgraphs which are large in some finer sense than \textit{complete}. Namely:

\begin{definition} A graph $G$ is \textnormal{$\kappa$-connected} if it remains connected after the deletion of any fewer than $\kappa$ vertices.
\end{definition} 

Our question should at this point be clear. Where formality is necessary, we will denote a graph $G$ as an ordered pair $(vertices,\,edges)$. The \textit{size} of a graph is the cardinality of its vertex-set. For cardinal numbers $\kappa$, $\lambda$, $\mu$, $\nu$, write
$$\nu\to_{\kappa\text{-}c}(\mu)^2_\lambda$$
if every coloring of the edges of the complete graph on $\nu$ into $\lambda$ many colors contains some size-$\mu$ monochromatic subgraph which is $\kappa$-connected. More formally:
\begin{definition}
$\nu\to_{\kappa\text{-}c}(\mu)^2_\lambda$ if and only if for every $c:[\nu]^2\rightarrow\lambda$ there exists a $\xi<\lambda$ and $X\in[\nu]^\mu$ such that the graph $(X,c^{-1}(\xi)\cap[X]^2 )$ is $\kappa$-connected.
\end{definition}

\begin{unnumberedquestion} For which cardinals does the relation $\nu\to_{\kappa\text{-}c}(\mu)^2_\lambda$ hold?
\end{unnumberedquestion}

\section{Main results}

Note at the outset that $\kappa$-connectedness is a well-studied notion, not least for its evident relevance to network design; it dates at least to Menger's 1927 \cite{menger} (see also \cite{mader}). Observe as well that it articulates a number of graph theory's most basic concerns:\begin{enumerate}
\item A graph is 1-connected if and only if it is connected. \item A graph is 2-connected if and only if each of its edges belongs to a cycle. \item The only $\mu$-connected graph on any finite $\mu$ is the complete one.\end{enumerate}
By this last point,
\[\tag{$**$}\text{when $\mu$ is finite,
$\nu\to_{\mu\text{-}c}(\mu)^2_\lambda$ is simply the Ramsey relation $\nu\to(\mu)^2_\lambda$.}\] Hence for any finite $\lambda$ and $\kappa\leq\mu$,
$$\nu\to_{\kappa\text{-}c}(\mu)^2_\lambda$$
for some $\nu\leq R(\mu;\lambda)$. See \cite{matula} for much finer bounds on the least such $\nu$.

For infinite $\mu$, a $\mu$-connected graph on $\mu$ is no longer necessarily complete; such graphs nevertheless play a sufficiently critical role in the theory to merit a name and notation all their own:

\begin{definition} A graph $G=(V,E)$ is \textnormal{highly connected} if it remains connected after the deletion of any fewer than $|V|$ vertices. Write $\nu\to_{hc}(\mu)^2_\lambda$ if and only if $\nu\to_{\mu\text{-}c}(\mu)^2_\lambda$, i.e., if every coloring of the edges of the complete graph on $\nu$ into $\lambda$ many colors contains some size-$\mu$ monochromatic subgraph which is highly connected.
\end{definition}
\begin{observation}\label{obvious}$\big[\nu\to(\mu)^2_\lambda\big]\Rightarrow\big[\nu\to_{hc}(\mu)^2_\lambda\big]\Rightarrow\big[\nu\to_{\kappa\text{-}c}(\mu)^2_\lambda\big]$ for any $\kappa\leq\mu$.
\end{observation}

In light of ($**$) and the following proposition, we might view $\to_{hc}$ as a more satisfactory generalization of the positive Ramsey relations of ($*$) to the uncountably infinite:

\begin{proposition}\label{firstpositiverelation} If $\mu$ is an infinite cardinal and $k$ is a natural number then $\mu\to_{hc}(\mu)^2_k$.
\end{proposition}

\begin{proof} Given a coloring $c:[\mu]^2\to k$, let $\mathcal D$ be a uniform ultrafilter on $\mu$ and define $f:\mu\to k$ by 
$$f(\alpha)=i\ \text{ if and only if }A_\alpha=\{\beta\in\mu\,:\,c(\{\alpha,\beta\})=i\}\in\mathcal D,$$
and let $i<n$ be such that the set $B=\{\alpha\in\mu: f(\alpha)=i\}$ is in $\mathcal D$. The set $B$ is highly connected. This is because for any $\alpha,\beta$ in $B$, the set $A_\alpha\cap A_\beta\cap B$ is in $\mathcal{D}$ and hence has cardinality $\mu$. Any $\gamma$ in $A_\alpha\cap A_\beta\cap B$ connects $\alpha$ and $\beta$ via the edge-colorings $c(\{\alpha,\gamma\})=c(\{\beta,\gamma\})=i$.
\end{proof}
 
The situation is considerably more complicated for infinitely many colors. Henceforth we will assume more set-theoretic background of the reader; we will focus as well on the relation $\to_{hc}$. This relation is subtle and significant in its own right, and we will tend to treat the finer relations $\to_{\kappa\text{-}c}$ as secondary, as mainly grading its failure. 

Perhaps the earliest result along these lines is Erd\H{o}s and Kakutani's theorem \cite{erdos-kakutani} that the complete graph on an infinite cardinal $\mu$ can be partitioned into $\lambda$ many trees if and only if $\mu\leq\lambda^{+}$. In consequence, the relation $\lambda^{+}\to_{hc}(\lambda^{+})^2_{\lambda}$ fails in the strongest possible respect: $\lambda^{+}\not\to_{2\text{-}c}(\lambda^{+})^2_{\lambda}$.

\begin{observation} $\mu\to_{1\text{-}c}(\mu)_\lambda^2$ holds for any cardinal $\lambda$  less than the cofinality of $\mu$.
\end{observation}

\begin{observation} The Erd\H{o}s-Kakutani coloring shows even that $\lambda^{+}\not\to_{hc}(\mu)^2_{\lambda}$ for any $\mu\geq 3$. \end{observation}

Alternately, $\lambda^{+}\not\to_{hc}(\lambda^{+})^2_{\lambda}$ may be viewed as an instance of the following proposition, inspired by the \emph{Sierpi\'nski coloring} of \cite{sierpinski}:

\begin{proposition} \label{negative} If $\mu\leq 2^\lambda$ then $\mu\not\to_{hc}(\mu)^2_\lambda$.
\end{proposition}

\begin{proof} Let $\{\eta_\alpha: \alpha<\mu\}\subseteq 2^\lambda$ be a collection of pairwise distinct functions. Given $\alpha\neq\beta <\mu$ let 
$$\triangle_{\alpha,\beta}=\min\{\xi<\lambda: \eta_\alpha(\xi)\neq\eta_\beta(\xi)\}.$$
Then define $c:[\mu]^2\to \lambda\times 2$ by
$$c(\{\alpha,\beta\})= (\triangle_{\alpha,\beta}, i)\text{ if and only if } \alpha<\beta\text{ and } \eta_\alpha(\triangle_{\alpha,\beta})=i.$$
Aiming for a contradiction assume that  $A\subset \mu$ is highly connected in color $(\xi,i)$.  Let $\beta\in A$ be such that that $\eta_\beta(\xi)\neq i$ (there is such a $\beta$ since $A$ contains an edge). Let $B=A\cap\beta$. Then $\beta$  has no adjacent edges in $(A\setminus B, c^{-1}((\xi,i))\cap [A\setminus B]^2)$, a contradiction.
\end{proof}

By the following observation, Proposition \ref{negative} says even more.

\begin{lemma}\label{lemma} Let $\mu$ be the cofinality of $\nu$. Then $\mu\not\to_{hc}(\mu)_\lambda^2$ implies that $\nu\not\to_{hc}(\nu)_\lambda^2$.
\end{lemma}
\begin{proof} Let $c:[\mu]\rightarrow\lambda$ witness that $\mu\not\to_{hc}(\mu)_\lambda^2$. Let $\nu$ be the disjoint union of $\mu$ many sets $\nu_\alpha$, each of strictly smaller cardinality than $\nu$. A coloring $d:[\nu]^2\rightarrow\lambda$ for which $d(\xi,\eta)=c(\alpha,\beta)$ if $\xi\in\nu_\alpha$ and $\eta\in\nu_\beta$ and $\alpha\neq\beta$ will witness that $\nu\not\to_{hc}(\nu)_\lambda^2$.
\end{proof}
The core relation in Proposition \ref{negative} is $2^\lambda\not\to_{hc}(2^\lambda)_\lambda^2$. This is sharp, in the sense that the relation $(2^\lambda)^{+}\to_{hc}(2^\lambda)_\lambda^2$ does hold:\footnote{(Compare the Erd\H{o}s-Rado relation $(2^\lambda)^{+}\to(\lambda^{+})_\lambda^2$. It too is sharp, in the sense that $2^\lambda\not\to(\lambda^{+})_\lambda^2$, as witnessed by the Sierpi\'nski coloring.)}
\begin{proposition}\label{muplus} If $\mu=\mu^\lambda$ then $\mu^+\to_{hc}(\mu)^2_\lambda$.
\end{proposition}

\begin{proof} Given a coloring $c:[\mu^+]^2\to \lambda$, let 
$\{M_\varepsilon: \varepsilon\leq\mu\}$ be a continuous chain of size-$\mu$ elementary submodels of some large enough $H(\theta)$ such that
\begin{enumerate}
\item $\mu+1\cup\{c\}\subseteq M_0$,
\item $[M_{\varepsilon+1}]^\lambda\subseteq M_{\varepsilon+1}$ for every $\varepsilon<\mu$, and such that
\item every formula $\varphi \in L_{\mu^+,\,\mu^+}(\in)$ satisfiable in $H(\theta)$ is satisfiable in $M_{\varepsilon+1}$ for every $\varepsilon<\mu$.
\end{enumerate}
Let $\delta_\varepsilon=M_\varepsilon\cap \mu^+\in\mu^+$ for every $\varepsilon\leq \mu$.
The sequence  $\{\delta_\varepsilon: \varepsilon\leq\mu\}$ is continuous and strictly increasing.

The key observation is the following:

\begin{claim} \label{main} For every $\varepsilon<\mu$ there is an $i(\varepsilon)<\lambda$ such that $c(\{\alpha,\delta_\mu\})=c(\{\beta,\delta_\mu\})=i(\varepsilon)$ for some $\alpha,\beta<\delta_\varepsilon$ and such that for every such $\alpha$, $\beta$, the set 
$$\{\gamma\in[\delta_\varepsilon, \delta_{\varepsilon+1}): c(\{\alpha, \gamma\})=c(\{\beta,\gamma\})=c(\{\gamma,\delta_\mu\})=i=i(\varepsilon)\}$$
is unbounded in $\delta_{\varepsilon+1}$.
\end{claim}

\begin{proof} Aiming towards a contradiction assume that the claim fails for some $\varepsilon<\mu$. In other words, for every $i<\lambda$ there are $\alpha_i, \beta_i<\varepsilon$ such that $c(\{\alpha_i,\delta_\mu\})=c(\{\beta_i,\delta_\mu\})=i(\varepsilon)$, yet the set
$$\Gamma_i=\{\gamma\in[\delta_\varepsilon, \delta_{\varepsilon+1}): c(\{\alpha_i, \gamma\})=c(\{\beta_i,\gamma\})=c(\{\gamma,\delta_\lambda\})=i=i(\varepsilon)\}$$
is bounded in $\delta_{\varepsilon+1}$.
As $[M_{\varepsilon+1}]^\lambda\subseteq M_{\varepsilon+1}$, the cofinality of $\delta_{\varepsilon+1}$ is bigger than $\lambda$ so there is common upper bound $\zeta<\delta_{\varepsilon+1}$ for all $\Gamma_i$, $i<\lambda$.

Consider the conjunction $\varphi$ of the formula $\gamma>\zeta$ with the formulas $c(\{\alpha_i, \gamma\})=c(\{\beta_i,\gamma\})=i$, where $i$ ranges below $\lambda$. Then $\varphi$ is satisfiable in $H(\theta)$, $\delta_\mu$ being the witness, so by (3) some $\gamma\in [\zeta, \delta_{\varepsilon+1})$ witnesses its satisfaction in $M_{\varepsilon+1}$. Let $j=c(\{\gamma,\delta_\lambda\})$. Then $\gamma\in \Gamma_j$ and $\gamma >\zeta$, contradicting the fact that $\Gamma_j$ was bounded by $\zeta$.
\end{proof}

Let $i<\lambda$ be such that the set $W=\{ \varepsilon: i(\varepsilon)=i\}$ has size $\mu$
(such an $i$ exists by assumption: $\mu=\mu^\lambda$ implies that the cofinality of $\mu$ is greater than $\lambda$), and let
$A=\bigcup\{A_\varepsilon:\varepsilon \in W\}$, where
$$A_\varepsilon=\{\alpha\in[\delta_\varepsilon,\delta_{\varepsilon+1}): c(\{\alpha,\delta_\mu\})=i\}.$$
$A$ is then a subset of $\mu^+$ of size $\mu$ (in fact, $A_\varepsilon$ has size $\mu$ for every 
$\varepsilon \in W$). We claim that it is highly connected in the color $i$. To see this it suffices to prove that
if $\alpha$ and $\beta$ are distinct elements of $A$ then the set
$$\{\gamma\in A: c(\{\alpha,\gamma\})=c(\{\beta,\gamma\})=i\}$$
has size $\mu$. To see this, let $\varepsilon_1$ and $\varepsilon_2$ be elements of $W$ such that
$\alpha\in A_{\varepsilon_1}$ and $\beta\in A_{\varepsilon_2}$, and let $\varepsilon$ be an element of $W\setminus (\varepsilon_1+1\cup\varepsilon_2+1)$.  We wish to find a $\gamma\in [\delta_\varepsilon, \delta_{\varepsilon+1})$ such that
 $$c(\{\alpha,\gamma\})= c(\{\beta, \gamma\})=c(\{\gamma,\delta_\mu\})=i=i(\varepsilon).$$
Such a $\gamma$ exists by Claim \ref{main}.
\end{proof}

A number of questions now come into focus. Most immediate among them is:

\begin{question} \label{bigquestion} For $\lambda$ an infinite cardinal, what is the least cardinal $\mu$ for which it is consistent with the ZFC axioms that $\mu\to_{hc}(\mu)^2_{\lambda}$?
\end{question}

By Observation \ref{obvious}, $\mu\to_{hc}(\mu)^2_\lambda$ holds whenever $\mu$ is weakly compact, for any $\lambda<\mu$. As we have seen, though, $\to_{hc}$ holds in many cases where the classical arrow fails; hence we might reasonably hope for $\mu\to_{hc}(\mu)^2_\lambda$ on much smaller $\mu$. Necessarily, $2^\lambda$ must be smaller than any such $\mu$, by Proposition \ref{negative}. Is this alone enough? 

No. By the following, any instance of $\mu\to_{hc}(\mu)^2_{\lambda}$ will involve large cardinal assumptions.

\begin{definition}
For regular uncountable $\mu$, the principle $\square(\mu)$ is the assertion that there exists a sequence $\mathcal{C}=\langle C_\alpha\,|\,\alpha\in\mu\rangle$ such that\begin{itemize}
\item $C_\alpha$ is a closed unbounded subset of $\alpha$, for each $\alpha$. 
\item $C_\beta\cap \alpha=C_\alpha$, for every limit point $\alpha$ of $C_\beta$. 
\item No club $C\subseteq\mu$ satisfies $C\cap \alpha=C_\alpha$ at every limit point $\alpha$ of $C$. 
\end{itemize}
If in addition the following holds, we will call $\mathcal{C}$ a \textit{$\lambda$-stationary $\square(\mu)$-sequence}:
\begin{itemize}
\item The set $\{\alpha\in\mu\,|\,\text{otp}(C_\alpha)=\lambda\}$ is stationary in $\mu$.
\end{itemize}
\end{definition}

The following is immediate from \cite{jensen} together with Lemma 7.2.2 of \cite{todorcevic-walks}.

\begin{theorem} If $\lambda<\mu$ are infinite regular cardinals and $\mu$ is not Mahlo in the constructible universe, then there exists a $\lambda$-stationary $\square(\mu)$-sequence.
\end{theorem}

\begin{proposition}\label{square} If there exists a $\lambda$-stationary $\square(\mu)$-sequence, then $\mu\not\to_{hc}(\mu)^2_\lambda$.
\end{proposition}

\begin{proof} The ``bad'' coloring $c:[\mu]^2\to\lambda$ will be $c(\alpha,\beta)=\rho^{\lambda}(\alpha,\beta)$, where $\rho^{\lambda}$ is Todorcevic's local rho function, defined in reference to some $\lambda$-stationary $\square(\mu)$-sequence $\mathcal{C}$. Readers are referred to \cite{todorcevic-walks} \textsection 7.2 for further information. The decisive features of $\rho^{\lambda}$ for our purposes are the following: for all $\alpha<\beta<\gamma<\mu$,
\begin{align}
\label{one} \rho^{\lambda}(\alpha,\beta) \leq & \:\text{max}\{\rho^{\lambda}(\alpha,\gamma),\rho^{\lambda}(\beta,\gamma)\},\text{ and} \\ \label{two} \rho^{\lambda}(\alpha,\gamma) \leq & \:\text{max}\{\rho^{\lambda}(\alpha,\beta),\rho^{\lambda}(\beta,\gamma)\}
\end{align}
In consequence, for all $\xi<\lambda$, the relation
$$\alpha<_\xi^\lambda\beta\;\text{ iff }\;\alpha<\beta\text{ and } \rho^{\lambda}(\alpha,\beta)\leq\xi$$ is a tree-ordering on $\mu$. By our assumptions about $\mathcal{C}$ and  Lemma 7.2.9 of \cite{todorcevic-walks}, none of the orderings $<_\xi^\lambda$ contains a chain of length $\mu$.

Now suppose towards contradiction that $A\in [\mu]^\mu$ is highly connected in the color $\xi$. By the above, there exist $\alpha<\beta$ in $A$ with $\alpha\nless_\xi^\lambda\beta$. By highly-connectedness, some color-$\xi$ path $\alpha=\alpha_0$ to $\alpha_1$ to $\dots$ to $\alpha_j$ to $\alpha_{j+1}=\beta$ must connect $\alpha$ and $\beta$ in $A\backslash\alpha$. It then follows from successive applications of (\ref{one}) and (\ref{two}) above that
$$\rho^{\lambda}(\alpha,\beta)\leq\:\max_{i\leq j}\rho^{\lambda}(\alpha_i,\alpha_{i+1})=\xi$$ This implies that $\alpha<_\xi^\lambda\beta$, a contradiction.
\end{proof}

\begin{corollary} It is consistent with the ZFC axioms --- and even with ZFC+GCH --- that $\mu\not\to_{hc}(\mu)_\lambda^2$ for all infinite cardinals $\lambda<\mu$.
\end{corollary}
\begin{proof} By Proposition \ref{square}, the relation $\mu\not\to_{hc}(\mu)_\lambda^2$ holds for any infinite regular cardinals $\lambda<\mu$ in a model of ZFC+(V=L)+``there exist no Mahlo cardinals.'' It will hold then for any singular $\mu$ by Lemma \ref{lemma}. Observe finally that if $\mu\not\to_{hc}(\mu)_\lambda^2$ failed for any singular $\lambda$, it would fail as well for some smaller regular $\lambda$, contradicting our premise.
\end{proof}

We now show in the opposite direction that, assuming  the existence of a weakly compact cardinal above some $\mu>\lambda$, it is consistent with the ZFC axioms that $2^{\mu}\to_{hc}(2^{\mu})^2_{\lambda}$.

Instrumental in the argument is the following ``two-dimensional delta system lemma'' of more general utility.

\begin{definition} A family of sets $\mathcal{A}$ is a  \textnormal{$\Delta$-system} if there exists a fixed $r$ such that $a\cap b =r$ for every distinct $a$ and $b$ in $\mathcal{A}$. This $r$ is called the \textnormal{root} of the $\Delta$-system.
\end{definition}

\begin{lemma} \label{deltasystemsystem}
Let $\nu$ be weakly compact and let $\mu$ be less than $\nu$. Then for any family $\{u_{\alpha,\beta}\,:\,\alpha<\beta<\nu\}\subseteq [\nu]^{\leq\mu}$ there exists a $B\in [\nu]^\nu$ such that:
\begin{enumerate} 
\item For each $\alpha\in B$, the set $\{ u_{\alpha,\beta}\,:\,\beta\in B\,\backslash\,(\alpha+1)\}$ is a $\Delta$-system, with root $V_\alpha^{+}$.
\item For each $\beta\in B$, the set $\{ u_{\alpha,\beta}\,:\,\alpha\in B\cap\beta\}$ is a $\Delta$-system, with root $V_\beta^{-}$.
\item The sets $\{V_\alpha^{+}\,:\,\alpha\in B\}$ and $\{V_\alpha^{-}\,:\,\alpha\in B\}$ and $\{V_\alpha^{-}\cup V_\alpha^{+}\,:\,\alpha\in B\}$ each form $\Delta$-systems.
\item The elements of the set $\{u_{\alpha,\beta}\,\backslash(V_\alpha^{+}\cup V_\beta^{-})\,:\,\alpha<\beta\textnormal{ in }B\}$ are pairwise disjoint.
%\item Within each of the following collections, all elements are order-isomorphic:\begin{enumerate}
%\item $\{u_{\alpha,\beta}\,:\,\alpha<\beta\textnormal{ in }B\}$,
%\item $\{V_\alpha^{+}\,:\,\alpha\in B\}$,
%\item $\{V_\alpha^{-}\,:\,\alpha\in B\}$,
%\item $\{V_\alpha^{-}\cup V_\alpha^{+}\,:\,\alpha\in B\}$.
%\end{enumerate}
\end{enumerate}
\end{lemma}

In what follows, the relation $\alpha<\beta$ will sometimes be left implicit; it is assumed to hold in any expression conjoining $\alpha$ and $\beta$.

\begin{proof}
By the weak compactness of $\nu$, we may begin by assuming all $u_{\alpha,\beta}$ to be of the same order-type.

Define the coloring $d:[\nu]^4\to H(\mu^+)$ as follows: for any increasing $a=\{\alpha_0, \alpha_1,\alpha_2, \alpha_3\}$ let 
 $$U_a=\,a\,\cup\bigcup_{j<k<4}u_{\alpha_j,\alpha_k}.$$ 
and let
$$d(a)=\langle otp(U_a), \langle\xi_i:i<4\rangle, \langle v_{j,k}:j<k<4\rangle\rangle,$$
so that if $h$ is the unique order-isomorphism between $U_a$ and $otp(U_a)$, then
\begin{itemize}
\item $h(\alpha_i)=\xi_i$ for every $i<4$, and
\item $h[u_{\alpha_j,\alpha_k}]=v_{j,k}$ for every $j<k<4$.
\end{itemize}
By the weak compactness of $\nu$, there exists a $d$-monochromatic $A\in [\nu]^\nu$. We argue most of the lemma for this set $A$, thinning to a $B\in[A]^\nu$ only later if necessary.

Consider $\alpha<\beta<\gamma<\delta<\varepsilon$ in $A$. Since $d(\alpha,\beta,\gamma,\varepsilon)=d(\alpha,\beta,\delta,\varepsilon)=d(\alpha,\gamma,\delta,\varepsilon)$,
\[\tag{$\dagger$}\xi\in u_{\alpha\beta}\cap u_{\alpha\gamma}\,\;\Leftrightarrow\,\;\xi\in u_{\alpha\beta}\cap u_{\alpha\delta}\,\;\Leftrightarrow\,\;\xi\in u_{\alpha\gamma}\cap u_{\alpha\delta}\]
As $\varepsilon$ was arbitrary, this implies item (1) of the lemma. We might usefully note more: a $\xi$ as in ($\dagger)$ must sit at the same relative height in each $u_{\alpha,\beta}$, for $\beta\in A\,\backslash(\alpha+1))$. Pigeonhole arguments together with $(\dagger)$ then ensure that any lesser elements of $u_{\alpha,\beta}$ also fall in the root $V_\alpha^{+}$ of the $\Delta$-system $\{ u_{\alpha,\beta}\,:\,\beta\in A\,\backslash\,(\alpha+1)\}$. In other words, $V_\alpha^{+}$ is an initial segment of each such $u_{\alpha,\beta}$.

Item (2) of the lemma is similar. (Pigeonhole arguments are not available in this case, hence the root $V_\beta^{-}$ is not so easily characterized.)

To see that $\{V_\alpha^{+}\,:\,\alpha\in A\}$ forms a $\Delta$-system with root $r^{+}$, observe that
\begin{align*}
\xi\in V_\beta^{+}\cap V_\gamma^{+} & \;\Rightarrow\; \xi\in u_{\beta,\varepsilon}\cap u_{\gamma,\varepsilon}\textnormal{ for any }\varepsilon\in A\,\backslash(\gamma+1)\\& \;\Rightarrow\; \xi\in V_\varepsilon^{-}\\ & \;\Rightarrow\; \xi\in u_{\alpha,\varepsilon}\textnormal{ for any }\alpha\in A\cap\varepsilon
\end{align*}
As $\varepsilon$ is arbitrary, this implies that $\xi$ is in $V_\alpha^{+}$. As $\alpha$ is arbitrary, this completes the argument. 

The argument that $\{V_\alpha^{-}\,:\,\alpha\in A\}$ forms a $\Delta$-system with root $r^{-}$ is essentially identical (but may require the omission of the first two elements of $A$).

Finally, note that $otp(V_\alpha^{+})$, $otp(V_\beta^{+})$, $otp(V_\gamma^{-})$, $otp(V_\delta^{-})$ are all legible from $d(\alpha,\beta,\gamma,\delta)$. In consequence:
\begin{enumerate}
\item[(i)] $V_\beta^{+}\cap V_\gamma^{-}$ is of the same order-type for all $\beta<\gamma$ in $A$. Hence this intersection must be of the form $r^{+}\cap\, r^{-}$. Thin $A$ if necessary to a $B\in [A]^\nu$ such that $V_\beta^+\backslash r^+ \cap V_\gamma^{-}\backslash r^-=\varnothing$ for all $\beta\geq\gamma$ in $B$. Then $\{V_\alpha\,:\,\alpha\in B\}$ forms a $\Delta$-system with root $r=r^{-}\cup r^{+}$.
\item[(ii)] The family $\{u_{\alpha,\beta}\,\backslash(V_\alpha^{+}\cup V_\beta^{-})\,:\,\alpha<\beta\textnormal{ in }B\}$ is pairwise disjoint. For by the homogeneity of $B$, any $\xi$ in $(u_{\beta,\gamma}\backslash V_\beta^{+}\cup V_\gamma^{-})\cap (u_{\alpha,\delta}\backslash V_\alpha^{+}\cup V_\delta^{-})$ is necessarily also in $(u_{\beta,\gamma}\backslash V_\beta^{+}\cup V_\gamma^{-})\cap (u_{\alpha,\varepsilon}\backslash V_\alpha^{+}\cup V_\varepsilon^{-})$ for any $\varepsilon\neq\delta$ in $B$. But this implies that $\xi\in V_\alpha^+$, a contradiction. Similarly for any other configuration of $\alpha,\beta,\gamma$, and $\delta$. 
\end{enumerate}
These establish items (3) and (4) of the lemma.
\end{proof}

\begin{remark} Two further features of the above system will be useful below:
\begin{itemize}
\item $V_\alpha^{+}$ are all of the same order-type, for $\alpha\in B$. Similarly for $V_\alpha^{-}$.
\item As each $\alpha$ in $B$ sits at the same distinguished relative location in $V_\alpha:= V_\alpha^{-}\cup V_\alpha^{+}$ we have $|\{otp(V_\alpha)\,:\,\alpha\in A\}|=1$ as well.
\end{itemize}
\end{remark}

\begin{proposition}\label{weaklycompact1} Let $\nu$ be a weakly compact cardinal, and let $\lambda<\mu=\mu^{<\mu}<\nu$ be given.
Then there is a cardinal-preserving forcing $\mathbb P$ such that 
$$\Vdash_{\mathbb P}\text{`` }2^\mu=\nu\text{ and }\nu\to_{hc}(\nu)^2_\lambda\text{''}.$$
\end{proposition}

\begin{proof} Let $\mathbb P$ be the forcing for adding $\nu$ many $\mu$-Cohen subsets of $\nu$, i.e.,
$$\mathbb P=\{p: p\text{ is a partial function from }\nu\text{ to }2\text{ of size less than }\mu\}$$
reverse-ordered by extension. By assumption, $\mathbb P$ has the $\mu^+$-c.c; it is evidently $\mu$-closed as well, and consequently preserves cardinals. By standard arguments, $\Vdash_{\mathbb P}``2^\mu=\nu$''.

Let $\dot c$ be a $\mathbb P$-name such that $\Vdash_{\mathbb P}``\dot c:[\nu]^2\to \lambda$''. For every $\alpha<\beta<\nu$ let $\mathcal{A}_{\alpha,\beta}:=\{p_{\alpha,\beta,\xi}:\xi<\mu\}$ be a maximal antichain in $\mathbb P$
with corresponding $\{i_{\alpha,\beta,\xi}:\xi<\mu\}$
 such that $p_{\alpha,\beta,\xi}\Vdash `` \dot c (\{\alpha,\beta\})=i_{\alpha,\beta,\xi}$''. Let 
 $$u_{\alpha,\beta}=\{\alpha,\beta\}\cup \bigcup_{\xi<\mu} \text{dom}(p_{\alpha,\beta,\xi})$$
Let $\{\gamma_{\alpha,\beta,\eta}:\eta<\varepsilon_{\alpha,\beta}\}$ enumerate $u_{\alpha,\beta}$ in increasing order. Define a relation $E$ on $[\nu]^2$ by declaring $\{\alpha_1,\beta_1\} \ E \ \{\alpha_2,\beta_2\}$ if and only if
 \begin{enumerate}
 \item $\varepsilon_{\alpha_1,\beta_1}=\varepsilon_{\alpha_2,\beta_2}$,
 \item $\alpha_1=\gamma_{\alpha_1,\beta_1,\eta}$ if and only if $\alpha_2=\gamma_{\alpha_2,\beta_2,\eta}$,
 \item $\beta_1=\gamma_{\alpha_1,\beta_1,\eta}$ if and only if $\beta_2=\gamma_{\alpha_2,\beta_2,\eta}$,
\item $\{\eta\,:\,\gamma_{\alpha_1,\beta_1,\eta}\in\text{dom}(p_{\alpha_1,\beta_1,\xi})\}=\{\eta\,:\,\gamma_{\alpha_2,\beta_2,\eta}\in\text{dom}(p_{\alpha_2,\beta_2,\xi})\}$ for every $\xi<\mu$,
 \item $p_{\alpha_1,\beta_1,\xi}(\gamma_{\alpha_1,\beta_1,\eta})=p_{\alpha_2,\beta_2,\xi}(\gamma_{\alpha_2,\beta_2,\eta})$ for every $\eta$ as in (4) and $\xi<\mu$, and
 \item $i_{\alpha_1,\beta_1,\xi}=i_{\alpha_2,\beta_2,\xi}$ for every $\xi<\mu$.
 \end{enumerate}
Clearly $E$ is an equivalence relation on $[\nu]^2$ with $2^\mu<\nu$ many equivalence classes. As $\nu$ is weakly compact, there exists an $A\in [\nu]^\nu$ such that $\{u_{\alpha,\beta}\,:\,\{\alpha,\beta\}\in[A]^2\}$ all falls in a single class. Further thin $A$ to a $B\subseteq A$ as in Lemma \ref{deltasystemsystem}. Write $\mathtt{r}$ for the root of the $\Delta$-system $\{V_\beta\,:\,\beta\in B\}$.

Observe that in this context the key terms of Lemma \ref{deltasystemsystem} take on more particular meanings: $V_\alpha^{+}$,  for example,  records exactly those coordinates at which some $p\in\mathcal{A}_{\alpha,\beta}$ and $q\in\mathcal{A}_{\alpha,\gamma}$ may disagree. The argument now proceeds in two steps; the uniformities of the family $B$ are important to each. In the first step, we extend any $q\in\mathbb{P}$ to an $s$ deciding the elements of some $\dot c$-homogeneous $X\in [\nu]^\lambda$. In the second step, genericity below $q$ propagates that homogeneity to highly connect a cofinal $Y\subseteq \nu$ in the forcing extension.

Therefore fix $q\in\mathbb{P}$ and a $W\subseteq B$ of order-type $\lambda^{+}+1$ such that $$\text{dom}(q)\cap\bigcup\, \{u_{\alpha,\beta}\,:\,\{\alpha,\beta\}\in [W]^2\}  \subseteq \mathtt{r}$$
Let $\{\varepsilon_\eta\,:\,\eta\leq\lambda^{+}\}$ enumerate $W$ in increasing order. We will recursively define ``colors'' $i_\eta\in\lambda$ and conditions $q_\eta$ below $q$ so that for all $\eta\leq\lambda^{+}$,
\begin{enumerate}
\item $\xi<\eta$ implies that $q_\xi\geq q_\eta$,
\item $q_\eta\Vdash\text{``}\dot{c}(\{\varepsilon_\xi,\varepsilon_{\lambda^+}\})=i_\xi\text{''}\,$ for all $\xi\leq\eta$,
\item $q_\eta\Vdash\text{``}\dot{c}(\{\varepsilon_\xi,\varepsilon_{\zeta}\})=i_\xi\text{''}\,$ for all $\xi<\zeta\leq\eta$, and
\item $q_\eta=q\cup\big[\bigcup_{\xi\leq\eta} p_{\varepsilon_\xi,\varepsilon_{\lambda^{+}}} \big]\cup\big[\bigcup_{\xi<\zeta\leq\eta} p_{\varepsilon_\xi,\varepsilon_\zeta}\big]$, where each $p_{\varepsilon_\xi,\varepsilon_\zeta}$ is a member of the antichain $\mathcal{A}_{\varepsilon_\xi,\varepsilon_\zeta}$. 
\end{enumerate}
To begin the construction, let $q_0$ be the union of $q$ with any compatible element of $\mathcal{A}_{\varepsilon_0,\varepsilon_{\lambda^{+}}}$. Suppose now that $q_\eta$ satisfying (1)-(4) have been constructed for every $\eta$ less than $\delta$. Then there exists for each such $\eta$ a $p_{\varepsilon_\eta,\varepsilon_{\lambda^{+}}}\in\mathcal{A}_{\varepsilon_\eta,\varepsilon_{\lambda^{+}}}$ so that $\bigcup_{\eta<\delta} p_{\varepsilon_\eta,\varepsilon_{\lambda^{+}}}\subseteq\bigcup_{\eta<\delta} q_\eta$. Let $p'_{\varepsilon_\eta,\varepsilon_{\delta}}$ denote their induced respective ``copies'' under the order-isomorphisms $\pi_\eta: u_{\varepsilon_\eta,\varepsilon_{\lambda^{+}}} \!\to u_{\varepsilon_\eta,\varepsilon_{\delta}}$. By arrangement, $$q'_\delta:=\big[\bigcup_{\eta<\delta} q_\eta\big]\cup\big[\bigcup_{\eta<\delta} p'_{\varepsilon_\eta,\varepsilon_{\delta}}\big]$$ is a function. Let $q_\delta$ be the union of $q'_\delta$ with any compatible element of $\mathcal{A}_{\varepsilon_\delta,\varepsilon_{\lambda^{+}}}$. The condition $q_\delta$ is as desired; in this fashion the construction proceeds.

For some $i\in\lambda$ the set $\{\eta\,:\,i_\eta=i\}$ is unbounded in $\lambda^{+}$. Let $X'$ collect its first $\lambda$ elements. Let $\bar{\eta}=\sup X'$ and let $s=q_{\bar{\eta}}$ and let $X=\{\varepsilon_\eta\,:\,\eta\in X'\}$. Clearly $s\Vdash\text{``}\dot{c}''[X]^2=\{i\}\text{''}$.  This completes the first of the steps described above.

For the second step, let $s_\varepsilon= s \restriction V_\varepsilon$ for each $\varepsilon\in X$. Let $V[G]$ be a forcing extension of $V$ by a $\mathbb{P}$-generic filter $G$; therein define the family
$$Y=\bigcup_{\varepsilon\in X} Y_{\varepsilon}$$
where 
\begin{align*} Y_\varepsilon=\{\alpha\in B\,:\, & \text{ there exists a }p\in G\text{ with }\text{dom}(p)\subseteq V_\alpha \text{ such that } \\ & \text{ the order-isomorphism } \pi:V_\alpha\to V_\varepsilon\text{ sends }p\text{ to }s_\varepsilon\}
\end{align*}
Lastly, let $q'=q\cup (s\restriction \mathtt{r})$.
\begin{claim}\label{claim} $q'\Vdash \text{``}(\dot {Y}, \dot c^{-1}(i)\cap [\dot{Y}]^2)\text{ is a highly connected graph of size }\nu \text{"}$.
\end{claim}
As $q$ was arbitrary, this claim will establish the proposition.
\begin{proof}[Proof of Claim \ref{claim}] Observe first that for each $\alpha\in\nu$ and $\varepsilon\in X$ the set $$D_\alpha^\varepsilon:=\{ p\,:\,p\Vdash\text{``}\dot{Y}_\varepsilon\not\subseteq\alpha\text{''}\}$$ is dense below $q'$. Hence $q'$ forces that each $Y_\varepsilon$ is unbounded in $\nu$.

Suppose now that $q''\leq q'$ forces that $\alpha$ and $\beta$ are in $Y$. Without loss of generality, $q''$ decides the witnesses to this fact as well, i.e., there exist some $p_\alpha, p_\beta\subseteq q$ and $\varepsilon(\alpha), \varepsilon(\beta)\in X$ such that the order-isomorphism $\pi:V_\alpha\to V_{\varepsilon(\alpha)}$ sends $p_\alpha$ to $s_{\varepsilon(\alpha)}$, and similarly for $\beta$. We'll show that for any $\gamma\in\nu$ there exists a $\delta>\gamma$ and $q'''\leq q''$ such that
\[\tag{$\ddagger$} q'''\Vdash\text{``}\dot{c}(\alpha,\delta)=\dot{c}(\beta,\delta)=i\text{''}\]
This will establish the claim. To that end, take $\varepsilon\in X\backslash (\text{max}\{\varepsilon(\alpha),\varepsilon(\beta)\}+1)$. By Lemma \ref{deltasystemsystem} there exists a $\delta\in B\backslash(\gamma+1)$ such that $u_{\alpha,\delta}\backslash V_\alpha$ and $u_{\beta,\delta}\backslash V_\beta$ are disjoint from $\text{dom}(q'')$. Extend $q''$ by ``copying'' $s\restriction (u_{\varepsilon(\alpha),\varepsilon}\cup u_{\varepsilon(\beta),\varepsilon})$ via the order-isomorphism to $u_{\alpha,\delta}\cup u_{\beta,\delta}$. Denote this extension by $q'''$. Our assumptions on $\alpha$ and $\beta$ ensure that $q'''$ is in fact a condition, and our assumptions on $E$ translate the relation $s\Vdash\text{``}\dot{c}(\varepsilon(\alpha),\varepsilon)=\dot{c}(\varepsilon(\beta),\varepsilon)=i\text{''}$ to ($\ddagger$), as desired.
\end{proof}
\end{proof}

\begin{corollary}\label{corollary}
Assuming the existence of a weakly compact cardinal, it is consistent with the ZFC axioms that $2^{\aleph_1}\to_{hc}(2^{\aleph_1})^2_{\aleph_0}$.
\end{corollary}

\section{Main questions}

We turn in conclusion to the most immediate instance of Question \ref{bigquestion}: 

\begin{question} \label{ctblquestion} What is the least cardinal $\mu$ for which it is consistent with the ZFC axioms that $\mu\to_{hc}(\mu)^2_{\aleph_0}$?
\end{question}

Corollary \ref{corollary} may be viewed as approximating to any of several possibilities. For example: the $\mu=2^{\aleph_1}$ of Corollary \ref{corollary} falls, in the forcing extension of Proposition \ref{weaklycompact1}, well below any weakly compact $\nu$, but it remains a regular limit cardinal. Therefore we may ask:

\begin{question} \label{inaccessible} Must the least cardinal $\mu$ for which it is consistent with the ZFC axioms that $\mu\to_{hc}(\mu)^2_{\aleph_0}$ be weakly inaccessible?
\end{question}

By Lemma \ref{lemma}, such a $\mu$ is necessarily regular; hence Question \ref{inaccessible} amounts to asking if such a $\mu$ may be a successor cardinal. If indeed it may be, then Corollary \ref{corollary} might be viewed instead as approximating to the following alternative to Question \ref{inaccessible}:

\begin{question}\label{aleph2}
Assuming whatever large cardinals may be necessary, is it consistent with the ZFC axioms that $\aleph_2\to_{hc}(\aleph_2)^2_{\aleph_0}$?
\end{question}

Reasoning from \cite{harringtonshelah} and \cite{lambiehanson} shows that if $\mu$ is the successor of a regular cardinal and $\lambda<\mu$, then the consistency strength of ``there exists no $\lambda$-stationary $\square(\mu)$ sequence'' is exactly a Mahlo cardinal. By Proposition \ref{square}, this gives a lower bound on the assumptions necessary to any affirmative answer to Question \ref{aleph2}. Again by Proposition \ref{square}, answers to any of the above questions entail further questions of consistency strength; most generally:

\begin{question}
What is the consistency strength of  the existence  of an uncountable cardinal $\mu$ such that $\mu\to_{hc}(\mu)^2_{\aleph_0}$?
\end{question}

By Proposition \ref{negative}, positive relations $\mu\to_{hc}(\mu)^2_\lambda$ will involve cardinal arithmetic assumptions as well. An affirmative answer to Question \ref{aleph2}, for example, would imply the continuum hypothesis. It is unclear if weaker relations would also. More particularly, the continuum hypothesis implies that $\aleph_2\to_{hc}(\aleph_1)^2_{\aleph_0}$, by Proposition \ref{muplus}. Is the reverse true? In other words:

\begin{question} Is the continuum hypothesis equivalent to the assertion that $\aleph_2\to_{hc}(\aleph_1)^2_{\aleph_0}$?
\end{question}

A question of a similar flavor is the following:

\begin{question}
Does $\aleph_2\to_{\aleph_1\text{-}c}(\aleph_2)^2_{\aleph_0}$?
\end{question}

\medskip

{\bf Acknowledgements.} The first author would like to thank Chris Lambie-Hanson for valuable discussion of the consistency strengths of failures of $\lambda$-stationary $\square(\mu)$ sequences to exist.

\end{document}